\documentclass{article}
\usepackage{color}
\usepackage{xcolor,xkeyval,pdfcolmk}
\usepackage[abbrev]{amsrefs}
\usepackage{amssymb,amsmath,amsthm}
\usepackage{fullpage}
\newcommand{\ie}{\emph{i.e.}}
\newcommand{\eg}{\emph{e.g.}}
\newcommand{\cf}{\emph{cf}}
\newcommand{\Real}{\mathbb{R}}
\newcommand{\Nat}{\mathbb{N}}

\newcommand{\esssup}{\mathop{\mathrm{ess\;\!sup}}}
\newcommand{\eps}{\varepsilon}
\newcommand{\Hilbert}{\mathcal{H}}
\newcommand{\Dom}{\mathfrak{D}}
\newcommand{\Ran}{\mathfrak{R}}
\newcommand{\dmn}{\mathsf{dim}}
\newcommand{\cdmn}{\mathsf{codim}}
\newcommand{\sii}{L^2}
\newcommand{\pa}{\partial}
\newtheorem{Theorem}{Theorem}
\newtheorem{Lemma}{Lemma}

\newtheorem{prop}{Proposition}
\theoremstyle{definition}
\newtheorem{definition}{Definition}
\newtheorem{Remark}{Remark}
\begin{document}
%
%-------%
% TITLE %
%-------%
%------------------------------------------%
%------------------------------------------%
\title{\textbf{\Large
Location of the essential spectrum in curved quantum layers% 
}}

\author{David Krej\v{c}i\v{r}{\'\i}k
\\
{\small Department of Theoretical Physics,}
\\
{\small Nuclear Physics Institute ASCR,}
\\
{\small 25068 \v{R}e\v{z}, Czech Republic}
\\ 
{\small krejcirik@ujf.cas.cz}
\bigskip \\
Zhiqin Lu
\\ 
{\small Department of Mathematics,}
\\ 
{\small University of California,}
\\
{\small Irvine, CA 92697, USA}
\\ 
{\small zlu@uci.edu}
} 

\date{\small 12 November 2012}
\maketitle

\begin{abstract}
\noindent
We consider the Dirichlet Laplacian
in tubular neighbourhoods of complete non-compact Riemannian 
manifolds immersed in the Euclidean space.
We show that the essential spectrum coincides with
the spectrum of a planar tube
provided that the second fundamental
form of the manifold vanishes at infinity and  the transport of the cross-section along the manifold
is  asymptotically parallel.
For low dimensions and codimension, the result applies
to the location of propagating states in nanostructures
under physically natural conditions.
\end{abstract}
%

%---------------------%
\section{Introduction}
%---------------------%
%
Consider a non-relativistic quantum particle constrained
to move in a vicinity~$\Omega$ of a submanifold~$\Sigma$ of
dimension~$\dmn$ in the Euclidean space of dimension $\dmn+\cdmn$. 
Throughout this paper, we assume that $\Omega$ is a fiber bundle over $\Sigma$ with fiber $\omega$, where $\omega$ is a bounded domain of $\mathbb R^{\cdmn}$.
Assuming hard-wall boundary conditions on~$\partial\Omega$,
the quantum Hamiltonian can be identified (in suitable units) 
with the Dirichlet Laplacian $-\Delta_D^\Omega$ on $\sii(\Omega)$.
The eigenfunctions and eigenvalues of $-\Delta_D^\Omega$ 
represent particle \emph{bound states} 
(\ie\ stationary solutions of the Schr\"odinger equation)
and their energies, respectively. 
However, for unbounded geometries the Schr\"odinger equation
also admits \emph{propagating/scattering states},
mathematically described by (a subset of) the essential spectrum 
of $-\Delta_D^\Omega$.
The results of the present paper 
(dealing with a more general setting) 
imply a precise location of the essential
spectrum $-\Delta_D^\Omega$ under the condition that~$\Sigma$ 
is a complete Riemannian manifold for which 
the \emph{second fundamental form vanishes at infinity}
and the translation of a uniform cross-section~$\omega$ of~$\Omega$
along~$\Sigma$ is \emph{asymptotically parallel}. 
%(in a sense precised below).

The conceptual model described above turns out to be 
a very good approximation for physical Hamiltonians
in \emph{nanostructures} 
(we refer to the survey~\cite{LCM} with many references).
Here the realistic cases $(\dmn,\cdmn)\in\{(1,1),(1,2),(2,1)\}$
are typically referred to as \emph{quantum strips},
\emph{quantum tubes} and \emph{quantum layers}, respectively.
The realizations with $\dim=1$ are also sometimes 
called \emph{quantum wires} and for all the systems 
the term \emph{quantum waveguides} is often used.
In order to emphasize the non-trivial nature of the geometry
of the underlying manifold, in this paper we adopt 
the term ``quantum layers''. 
However, arbitrary values $\dmn \geq 1$ and $\cdmn \geq 1$
are allowed here.
Alternative physical justifications for considering 
the quantum Hamiltonian $-\Delta_D^\Omega$
(including high values of~$\dmn$ and arbitrary values of~$\cdmn$)
are given in terms of effective models for molecular dynamics
and quantization on submanifolds
(we refer to a recent work~\cite{Wachsmuth-Teufel}
with many references).

Let us now assume that 
the second fundamental form of~$\Sigma$ vanishes at infinity
and that the transport of~$\omega$ along~$\Sigma$ 
is asymptotically parallel.
Then, using the Fermi coordinates based on~$\Sigma$,
it is possible to check that the layer~$\Omega$ represents 
a local deformation of the product manifold $\Sigma\times\omega$, 
in the sense that the difference of 
the two corresponding metrics becomes close at infinity.  
Since the essential spectrum is expected to be determined by
the behaviour of the metric at infinity only,
it is natural to \emph{conjecture} that 
\begin{equation}\label{exp1}
  \sigma_\mathrm{ess}(-\Delta_D^\Omega)
  = \overline{
  \bigcup_{k=1}^\infty 
  \big(
  E_k+\sigma_\mathrm{ess}(-\Delta_g)
  \big)
  }
  \,.
\end{equation}
Here~$-\Delta_g$ denotes the Laplace-Beltrami operator 
of~$\Sigma$ equipped with a Riemannian metric~$g$
and 
\begin{equation*}
  	E_1 < E_2 \leq E_3 \leq \dots
\end{equation*}
are the discrete eigenvalues 
of the cross-section Dirichlet Laplacian 
$-\Delta_D^\omega$ on $\sii(\omega)$. 
Moreover, since the vanishing of the second fundamental form
at infinity implies that~$\Sigma$ ``looks Euclidean'' 
on enlarging balls ``localized at infinity'',
it is \emph{expectable} that
\begin{equation}\label{exp2}
  \sigma_\mathrm{ess}(-\Delta_g)
  = [0,\infty)
  \,.
\end{equation}
Indeed, the right hand side is the essential spectrum 
of the Laplacian in $\Real^\dmn$.
Summing up, having~\eqref{exp1} and~\eqref{exp2},
one could conclude with the ultimate result
\begin{equation}\label{exp3}
  \sigma_\mathrm{ess}(-\Delta_D^\Omega)
  = [E_1,\infty)
  \,.
\end{equation}
This is a prototype of the results we establish in this paper.

%---------------------------------------------%
\subsection{Previous results and our strategy}
%---------------------------------------------%
%
It is not difficult to establish~\eqref{exp3} provided that
one is ready to impose some extra conditions about the decay of
the derivative of the second fundamental form of~$\Sigma$ at infinity.
Indeed, to show that the point $ \lambda \in [E_1,\infty)$ belongs to
the essential spectrum of $-\Delta_D^\Omega$, 
according to the Weyl criterion,
one needs to find a normalized sequence of functions~$\psi_n$ 
from the domain of $-\Delta_D^\Omega$ such that 
$(-\Delta_D^\Omega-\lambda)\psi_n$ converges strongly
to zero in $\sii(\Omega)$.
However, passing to the Fermi coordinates based on~$\Sigma$, 
$-\Delta_D^\Omega$ is unitarily equivalent to
a Laplace-Beltrami-type operator with a metric~$G$
depending explicitly on the second fundamental form of~$\Sigma$ and the transport of the cross-section along the manifold.
It is less obvious how to proceed without differentiating
the metric~$G$.
Let us also mention that the proof of the fact 
$\inf\sigma_\mathrm{ess}(-\Delta_D^\Omega)=E_1$
is considerably easier since one can apply variational tools.
Summing up, we understand the following two steps 
as the key ingredients to establish~\eqref{exp3} under the minimal
assumptions of the present paper:
\begin{enumerate}
\item[(i)]
prove~\eqref{exp2} for the Laplace-Beltrami operator $-\Delta_g$
of~$\Sigma$;
\item[(ii)]
apply a Weyl-type criterion modified for quadratic forms
to $-\Delta_D^\Omega$.
\end{enumerate}

\paragraph{\emph{Ad} (i).}
The situation $\dmn=1$ is elementary, 
since any complete non-compact one-dimensional manifold (curve)
is diffeomorphic to the real line with Jacobian one. 
Hence, \eqref{exp2}~holds trivially in this case. 
In the following introductory exposition, 
we therefore assume $\dmn \geq 2$.

In the past two decades, the essential spectra of  Laplacians 
on functions were computed for a large class of manifolds.
When the manifold~$\Sigma$ has a soul  
and the exponential map is a diffeomorphism, 
Escobar~\cite{E} and Escobar-Freire~\cite{ef} proved~\eqref{exp2}, 
provided that  the  sectional curvature is non-negative 
and the manifold satisfies some additional conditions.
Zhou~\cite{Z} proved that those ``additional conditions'' are superfluous.  
When the manifold has a pole, Li~\cite{jli} proved~\eqref{exp2},
if the Ricci curvature of the manifold is non-negative. 
Chen and the second author~\cite{C-L} proved 
the same result  when the radical sectional curvature is non-negative. 
Among the other results in his paper~\cite{donnelly-1}, 
Donnelly proved~\eqref{exp2}
for manifolds with non-negative Ricci curvature 
and Euclidean volume growth.

In 1997, Wang~\cite{wang} proved  that, 
if the Ricci curvature of a manifold~$\Sigma$ satisfies 
${\rm Ric}\,(\Sigma)\geq-\delta/r^2$, 
where~$r$ is the distance to a fixed point, 
and~$\delta$ is a positive number depending only on the dimension,  
then the $L^p$ essential spectrum of~$\Sigma$ 
is $[0,\infty)$ for any  $p\in[1,+\infty]$.
In particular, for a complete non-compact manifold 
with non-negative Ricci curvature, 
all $L^p$ spectra are $[0,\infty)$.  
In 2011, Lu-Zhou~\cite{Lu-Zhou_2011} generalized 
the  result  of Wang and proved that, 
if the Ricci curvature of~$\Sigma$ is non-negative at infinity, 
then all $L^p$ spectra are $[0,\infty)$.  
In both papers \cites{wang, Lu-Zhou_2011}, 
the paper of Sturm~\cite{sturm} is used in the essential way.
In particular, it follows from~\cite{Lu-Zhou_2011} 
that~\eqref{exp2} is satisfied for manifolds with 
Ricci curvature vanishing at infinity,
which cover the class of manifolds we are interested in.

\paragraph{\emph{Ad} (ii).}
In the past two decades,
there were many attempts to establish~\eqref{exp3}
in the context of quantum waveguides when $\dmn=1$.
The results typically required to impose    
some additional assumptions about the decay of curvatures 
of the curve~$\Sigma$ together with their derivatives at infinity,
so that the classical Weyl criterion could be applied.
Location of the essential spectrum~\eqref{exp3} 
under the mere vanishing of curvatures at infinity was
established for the first time by K\v{r}\'i\v{z} 
and the first author in the 2005 paper \cite{KKriz} 
for the $(1,1)$ case (quantum strips) 
and later generalized to arbitrary $\cdmn \geq 1$ in \cite{ChDFK}. 
The breakthrough results of \cites{KKriz,ChDFK} were achieved
thanks to a usage of the Weyl criterion adapted to quadratic forms.
According to this improved criterion, it is enough to look for
the singular sequence~$\psi_n$ in the form domain of $-\Delta_D^\Omega$
and the convergence of $(-\Delta_D^\Omega-\lambda)\psi_n$ to zero
should hold in the (weaker) topology of the dual of the form domain.
It seems that this version of the Weyl criterion is not well known;
we learnt it from a private discussion with Iftimie in 2002
\cite{Iftimie-private}
(see also \cite{DDi}*{Lem.~4.1} for the statement without proof).
For completeness, we present a short proof 
of the criterion in the appendix to this paper
(\cf~Theorem~\ref{Thm.Weyl.bis}).

Due to more complicated  geometrical and topological settings,
the cases of $\dmn \geq 2$ were much less studied during the last years.
The concept of quantum layers, in the $(2,1)$ case, 
was introduced in~\cite{DEK2}.
Under the quite restrictive hypothesis that~$\Sigma$ possesses a pole,  
the minimax principle was used in~\cite{DEK2}
to show that the essential spectrum of $-\Delta_D^\Omega$
is bounded from below by the energy~$E_1$,
provided that the principal curvatures of~$\Sigma$ vanish at infinity.
The hypothesis about the pole was later removed in~\cite{CEK},
where it was additionally shown, still by variational methods, that  
$\inf\sigma_\mathrm{ess}(-\Delta_D^\Omega)=E_1$.
The extension of the lower bound to higher-dimensional situations
$\dmn \geq 3$ and $\cdmn \geq 2$ was performed in \cites{LL2,LL1}.
The complete result~\eqref{exp3} was only established
in the thesis \cite{these}, where, however, additional assumptions
about the decay of the derivatives of the second fundamental form
of~$\Sigma$ at infinity were required.

\paragraph{ }%
In the present paper, we are eventually able to combine 
the Weyl criterion adapted to quadratic forms with 
the very recent result of~\cite{Lu-Zhou_2011} establishing~\eqref{exp2}
in order to conclude with the desired property~\eqref{exp3},
assuming only the mere vanishing of the second fundamental form
of~$\Sigma$ at infinity.
We believe that this is the first time that the weak 
Weyl's criterion is applied on manifolds and differential geometry. 
Moreover, we proceed in a much greater generality 
by considering the \emph{layer}~$\Omega$ 
as a quite arbitrary fiber bundle of~$\Sigma$
and the cross-section~$\omega$ of~$\Omega$ 
is allowed to ``rotate'' along~$\Sigma$.

%------------------------------------------------%
\subsection{The general setting and main results}
%------------------------------------------------%
%
%All manifolds in this paper are $\mathcal C^2$ manifolds.
Let $\dmn, \cdmn$ be positive integers.
Let~$\omega$ be a bounded open connected set in $\mathbb R^{\cdmn}$
(no regularity assumptions about~$\partial\omega$ are required). 
Let~$\Sigma$ be a complete non-compact Riemannian manifold  
of dimension $\dmn \geq 1$  with a Riemannian metric~$g$. 
We assume that $(\Sigma,g)$ is of class $\mathcal C^2$. 
A \emph{layer} is a fiber bundle 
and a Riemannian manifold $\pi: \Omega\to\Sigma$  with fiber $\omega$. 
Let  the  Riemannian metric of $\Omega$ be $G$.

We are interested in the situation when~$\Omega$ 
is a ``local deformation'' of the \emph{unperturbed layer}~$\Omega_0$,
the latter being defined as $\Sigma\times\omega$ equipped with
the metric~$G_0$ of the block form  
\begin{equation}\label{block}
  G_0 :=
  \begin{pmatrix}
    g & 0 \\
    0 & 1
  \end{pmatrix}
  \,.
\end{equation}
\begin{definition}\label{def1}
We say that  the metric $G$ is a \emph{local deformation} of $G_0$, 
if there exist a sequence $\{y_i\}$ of points on $\Sigma$ 
and a sequence $\{R_i\}$ of real numbers such that
\begin{enumerate}
\item $y_i\to\infty$ (i.e. $d(y_i,o)\to\infty$ for a fixed point $o\in\Sigma$) and $R_i\to\infty$ as $i\to\infty$;
\item $B_{y_i}(R_i)$, 
the geodesic balls centered at $y_i$ with radius $R_i$, 
are disjoint;
\item there is a diffeomorphism  $\pi^{-1}(B_{y_i}(R_i))=B_{y_i}(R_i)\times \omega$, under which 
\begin{equation}\label{local}
  \lim_{i\to\infty} 
  \esssup_{\pi^{-1}(B_{y_i}(R_i))} 
  |G-G_0| = 0
  \,,
\end{equation} where 
$|\cdot|$ denotes  norm under the metric $G_0$.

\end{enumerate}
\end{definition}

In particular, if  $\Omega$ is diffeomorphic to $\Sigma\times\omega$ and 
\begin{equation}\label{local-2}
  \lim_{R\to\infty} 
  \esssup_{(\Sigma\setminus B_R)\times\omega} 
  |G-G_0| = 0
  \,,
\end{equation}
where $B_R$ is the ball of radius $R$ with respect to a fixed point $o\in\Sigma$,
then $G$ is a local deformation of $G_0$. 
Note that in general, we do not require that $\Omega$ 
is diffeomorphic to $\Sigma\times\omega$.\\

Throughout  this paper, we  assume that~$\Sigma$ 
is an immersed surface in $\mathbb R^{\dmn+\cdmn}$.
We say that $\Sigma$ is \emph{asymptotically flat} if
the second fundamental form of $\Sigma$ tends to zero at infinity.
If the ``cross-section''~$\omega$ is not a ball centered
at the origin of~$\mathbb{R}^\cdmn$, we also need to ensure
that the transport of~$\omega$ along~$\Sigma$ 
is asymptotically parallel in a sense; 
this is formalized in Definition~\ref{def4} below
where such~$\Omega$ is called \emph{asymptotically flat immersed layer}.\\
 
Let~$-\Delta_G$ denote the Friedrichs extension on $\sii(\Omega,G)$
of the Laplace-Beltrami operator of $(\Omega,G)$
initially defined on $\mathcal C_0^\infty(\Omega)$.
For immersed layers, $-\Delta_G$~can be identified 
with the Dirichlet Laplacian~$-\Delta_D^\Omega$ discussed above.
There are many papers dealing with the existence and properties
of the discrete eigenvalues of~$-\Delta_G$; 
see, \eg, \cites{DEK2,CEK,EK3,LL1,LL2}.
In this paper, we focus on the essential spectrum.
Let us note that, by ``separation of variables''
(\cf~Section~\ref{Sec.proofs}), 
it is easy to locate the essential spectrum of
the unperturbed layer~$\Omega_0$: 
\begin{equation}\label{EssSpec0}
  \sigma_\mathrm{ess}(-\Delta_{G_0}) = [E_1,\infty)
\end{equation}
provided that~\eqref{exp2} holds
(which is the case if the Ricci curvature of~$\Sigma$ 
vanishes at infinity \cite{Lu-Zhou_2011},
in particular if~$\Sigma$ is asymptotically flat).

Our main results are the following two theorems.
\begin{Theorem}\label{thm3}
Let $\Omega$ be an immersed  layer (See Definition~\ref{def3}). 
Assume that $\Sigma$ is an asymptotically flat immersed submanifold. 
Then 
\[
\sigma_\mathrm{ess}(-\Delta_{G}) \subset [E_1,\infty)
  \,.
\]
In other words, the threshold of the essential spectrum satisfies
the lower bound
\[
\inf\,\sigma_\mathrm{ess}(-\Delta_{G})\geq E_1.
\]
\end{Theorem}
\begin{Theorem}\label{thm1}
Let $\Sigma$ be an asymptotically flat immersed submanifold.
Assume that $G$ is a local deformation of~$G_0$.  Then
\begin{equation*}
  \sigma_\mathrm{ess}(-\Delta_{G}) 
 \supset [E_1,\infty)
  \,.
  \end{equation*}
\end{Theorem}

The conditions stated in the theorems are the relevant ones.
It can be seen on simplest non-trivial situations 
when~$\Sigma$ is a curve.
Indeed, if~$\Sigma$ is a periodically curved curve in~$\Real^2$
(so that~$\Sigma$ is not asymptotically flat), 
then $\inf\sigma_\mathrm{ess}(-\Delta_{G})<E_1$;
see~\cite{KKriz}. 
On the other hand, if a non-circular~$\omega$ is periodically
rotated along a straight line in~$\Real^3$
(so that~$G$ is not a local deformation of~$G_0$), 
then $\inf\sigma_\mathrm{ess}(-\Delta_{G})>E_1$;
see~\cite{moi}.  
These examples demonstrate that both 
the curvatures of~$\Sigma$ and the transverse connection
must necessarily vanish at infinity in order to ensure that 
the essential spectrum coincides with
the (essential) spectrum~\eqref{EssSpec0} 
of the unperturbed layer~$\Omega_0$.
In this paper, we show that these natural conditions 
are actually sufficient for the location 
of the essential spectrum:
\begin{Theorem}\label{thm2}
Let~$\Omega$ be an asymptotically flat immersed layer 
(see Definition~\ref{def4}). 
Then
\begin{equation*}
  \sigma_\mathrm{ess}(-\Delta_{G}) 
  =  [E_1,\infty)
  \,.
\end{equation*}
\end{Theorem}

This last theorem is established as a corollary of 
Theorems~\ref{thm3} and~\ref{thm1} by noticing that
the Riemannian metric~$G$ of an asymptotically flat immersed   
layer is a local deformation of~$G_0$. 

%------------------------------------%
\subsection{The content of the paper}
%------------------------------------%
%
The paper is organized as follows.
In Section~\ref{Sec.pre} we provide precise definitions of various 
geometric and analytic objects we use in the paper;
in particular, we develop the notion of Fermi coordinates
for curved layers. Theorems~\ref{thm3}--\ref{thm2} 
are proved in Section~\ref{Sec.proofs}.
As mentioned above, the main ingredient in the proof
is the Weyl criterion adapted to quadratic forms.
Since we are not aware of the existence of this 
useful tool in the literature, we decided to state it
together with a short proof in Appendix~\ref{App}. 

%----------------------%
\section{Preliminaries}\label{Sec.pre}
%----------------------%
%
We extensively use the notion of \emph{Fermi coordinates}~\cite{gray},
which are the natural coordinates for tubular geometries.
We employ the convention of indexing local coordinates $x_m$ in~$\Sigma$
and Cartesian coordinates $u_\mu$ in~$\omega$ 
by Latin and Greek indices, respectively,  
the range of them being $m\in\{1,\dots,\dmn\}$ 
and $\mu\in\{\dmn+1,\dots,\cdmn\}$, respectively. 
Then the local coordinates $(x_1,\dots,x_\dmn,u_1,\dots,u_\cdmn)$
are the Fermi coordinate for the unperturbed layer~$\Omega_0$,
\ie~the product manifold $\Sigma\times\omega$ 
equipped with the metric~$G_0$.
Capital Latin indices are used for indexing the local coordinates
in $\Sigma\times\omega$, \ie\ $M,N\in\{1,\dots,\dmn+\cdmn\}$.
Einstein's summation convention is assumed throughout the paper.

In order to introduce a Fermi coordinate system on the layer~$\Omega$,
we need the following auxiliary result. 

\begin{Lemma}\label{lem1} 
Let $\Sigma$ be an immersed submanifold of $\mathbb R^{\dmn+\cdmn}$.
Let $T^\perp\Sigma$ be the normal bundle of $\Sigma$ in $\mathbb R^{\dmn+\cdmn}$. 
Assume that $\Sigma$ is asymptotically flat.
Then there exists a sequence $y_i\in \Sigma$ 
and a sequence $R_i\to\infty$ of positive numbers such that
\begin{enumerate}
\item the injectivity radius of $y_i$ is at least $R_i$;
\item on each $B_{y_i}(R_i)$, there exist local orthonormal  frames  $f_1,\cdots, f_{\cdmn}$ of the normal bundle such that the connection metrics of $\nabla^\perp$ with respect to these frames go to zero.

\item Let $ds^2$ be the pull back of the Riemannian metric of $\Sigma$ onto the tangent space $T_{y_i}\Sigma$ via the exponential map at $y_i$. Then we have
\[
(1-\eps) ds_0^2\leq ds^2\leq (1+\eps) ds_0^2
\]
on the ball of radius $R_i$ in $T_{y_i}\Sigma$, where $ds_0^2$ is the Euclidean metric of $T_{y_i}\Sigma$.

\end{enumerate}
\end{Lemma}
\begin{Remark}
Note that property~2  above implies that $\pi^{-1}(B_{y_i}(R_i))$ 
is diffeomorphic to $B_{y_i}(R_i)\times \omega$.
\end{Remark}
\begin{proof}
We pick any sequence $y_i\in\Sigma$, $y_i\to\infty$ as $i\to\infty$. 
We shall prove that the injectivity radius at $y_i$ goes to infinity. 
Since the curvature of~$\Sigma$ goes to zero at infinity, 
there exist $R_i\to\infty$ such that on the ball $B_{y_i}(R_i)$, there are no conjugate points of $y_i$ by Meyer's theorem~\cite{cheeger}*{page 27}. Therefore, in order to get the injectivity radius estimate, we need to prove that there are no short closed geodesic loops. 

Assume that $\sigma$ is the closed geodesic loop.
Then it can be regarded as a smooth curve in $\mathbb R^{\dmn+\cdmn}$ (except at  the point $y_i$). 
Let~$\kappa$ be the curvature of $\sigma$. 
Since the angle of the curve at $y_i$ is at most $\pi$,  we have (\cf~\cites{ku,milnor})
\[
\pi+\int_\sigma|\kappa|\geq 2\pi.
\]
On the other hand, since $\sigma$ is a geodesic line,  $\kappa$ is bounded by norm of the the second fundamental form. Therefore $\kappa\to 0$. By the above inequality, the length of $\sigma$ becomes very large as $y_i\to\infty$. Therefore there exist $R_i\to\infty$ such that the injectivity radius at $y_i$ is at least $R_i$.

Fixing $i$, without loss of generality, we assume $y_i$ is the origin of $\mathbb R^{\dmn+\cdmn}$ and the tangent space of $\Sigma$ at $y_i$ is the subspace
$(\mathbb R^\dmn,0)\subset \mathbb R ^{\dmn+\cdmn}$. Since the injectivity radius at $y_i$ is at least $R_i$. Let $\iota:\Sigma\to\mathbb R^{\dmn+\cdmn}$ be the inclusion and let $\exp_{y_i}$ be the exponential map at $y_i$. Let $\sigma=\iota\circ\exp_{y_i}$. Then we can write
$\sigma: \mathbb R^{\dmn}\to\Sigma\subset \mathbb R^{\dmn+\cdmn}$ by
\begin{equation}
\sigma(x_1,\cdots, x_{\dmn})=(x_1,\cdots,x_\dmn, \sigma_{\dmn+1},\cdots,\sigma_{\dmn+\cdmn}),
\quad \sum_j|x_j|^2\leq R_i^2,
\end{equation}
where $\sigma_{\dmn+1},\cdots,\sigma_{\dmn+\cdmn}$ are functions of $x_1,\cdots, x_\dmn$. Assume that the second fundamental form on $B_{y_i}(R_i)$ is less than $\eps$. Then we have
\begin{equation}\label{20}
g^{ik} g^{jl}\,h_{\alpha\beta}\frac{\pa^2\sigma_\alpha}{\pa x_i\pa x_j}\cdot\frac{\pa^2\sigma_\beta}{\pa x_k\pa x_l}(1+|\nabla\sigma_\alpha|^2)^{-1/2}(1+|\nabla\sigma_\beta|^2)^{-1/2}\leq \eps^2,
\end{equation}
where 
\[
g_{ij}=\delta_{ij}+\sum_\alpha\frac{\pa\sigma_\alpha}{\pa x_i}\cdot\frac{\pa\sigma_\alpha}{\pa x_j},
\]
and 
\[
h_{\alpha\beta}=\delta_{\alpha\beta}+\sum_i\frac{\pa\sigma_\alpha}{\pa x_i}\cdot\frac{\pa\sigma_\beta}{\pa x_i}
\]
are the metric matrices of $\Sigma$ and $T^\perp\Sigma$, respectively.

Let $x_1,\cdots,x_{\dmn}$ be the normal coordinate system at $y_i$. Then $\nabla\sigma_\alpha(y_i)=0$.
Let $\delta\leq R_i$ be the largest number such that 
\[
\sum_{j,\alpha}\left|\frac{\pa\sigma_\alpha}{\pa x_j}\right|^2\leq \frac 12
\]
on $B_{y_i}(\delta)$.
If $\delta<R_i$, then by the definition of $\delta$,  there is a point $y'\in\pa B_{y_i}(\delta)$ such that 
\[
\sum_{j,\alpha}\left|\frac{\pa\sigma_\alpha}{\pa x_j}\right|^2(y')=\frac 12.
\]
By~\eqref{20}, we have
\begin{equation}\label{21}
\sum_{j,k,\alpha}\left|\frac{\pa^2 \sigma_\alpha}{\pa x_j\pa x_k}\right|^2\leq (3/2)^{4n}\eps^2
\end{equation}
on $B_{y_i}(\delta)$. Since $\nabla \sigma_\alpha(y_i)=0$, by the mean value theorem, we have
\[
\frac 12 =\sum_{j,\alpha}\left|\frac{\pa\sigma_\alpha}{\pa x_j}\right|^2(y')\leq (3/2)^{4n}\,\dmn\cdot\cdmn\,\eps^2\delta^2.
\]
Thus $\delta\geq C\eps^{-1}$ for some constant $C$ depending only on the dimensions. If $\eps$ is sufficiently small, then we have $\delta=R_i$. 

By~\eqref{21}, we know that on $B_{y_i}(\sqrt{R_i})$, both the first and the second derivatives of $\sigma_\alpha$ goes to zero as $i\to\infty$.

Let
\[
n_\alpha:=\left(-\frac{\pa \sigma_\alpha}{\pa x_1},
\cdots, -\frac{\pa \sigma_\alpha}{\pa x_n}, 0,\cdots, 1,\cdots, 0\right)
\]
where $1$ is in the $(\dmn+\alpha)$-th place. Then $n_{\dmn+1},\cdots, n_{\dmn+\cdmn}$ are ``almost'' orthonormal on $B_{y_i}(R_i)$  given that the first derivatives are small. Using the Gram-Schmidt process, we obtain an orthonormal frame system $f_1,\cdots, f_{\cdmn}$. It is not hard to see that the connection $\nabla^\perp$ on $f_\alpha$ are bounded by the first and the second derivatives of $\sigma_\alpha$, hence go to zero as $i\to\infty$. This proves part 2 of the lemma, if we replace $R_i$ by $\sqrt{R_i}$, the latter  tending to infinity as $i\to\infty$. Part 3 of the lemma follows from the fact that $ds_0^2$ and $ds^2$ differ by the $\mathcal C^1$-norm 
of the vector-valued functions $\sigma_\alpha$.
\end{proof}
With the above lemma, we are able to generalize the Fermi coordinate systems to fiber bundles over $\Sigma$.
\begin{definition}
Consider the isometric immersion 
$\varSigma \to \mathbb{R}^{\dmn+\cdmn}$ of the 
base manifold.
For a local coordinate chart $(U, \phi)$ on $\varSigma$, 
we trivialize 
the normal bundle $T^{\bot}\varSigma$ over $U$ with an orthonormal 
frame 
$\{f_1, \cdots, f_\cdmn\}$.  Then for each $x \in U$,  
$(x,\xi) \in T_{x}^{\bot}\varSigma$ we define local coordinates 
$(x_1, \cdots, x_\dmn, u_1, \cdots, u_\cdmn)$, where
$\xi = \sum u_{\alpha}f_{\alpha}(x)$.  We call such a coordinate system
a \emph{Fermi coordinate system} on $T^{\bot}\varSigma$.   
\end{definition}

As a set, $T^\perp\Sigma$ is $\mathbb R^{\dmn+\cdmn}$. 
Moreover, there is an \emph{endpoint map} (\cf~\cite{milnor-2}*{page 32})
\[
p: T^\perp\Sigma\to\mathbb R^{\dmn+\cdmn}
\]
defined as follows: 
for any element $(x,v)$, where $x\in\Sigma$ 
and $v$ is a normal vector of $\Sigma$ at $x$, we have 
\[
p(x,v)=x+v.
\]

Under a Fermi coordinate system the entries in the metric tensor~$G$ 
can be expressed as
\begin{equation}\label{metric-coord}
G_{ij}dx_{i}dx_{j} 
+ G_{i\alpha}dx_{i}du_{\alpha} 
+ G_{\alpha i}du_{\alpha}dx_{i} 
+ G_{\alpha\beta}du_{\alpha}du_{\beta}. 
\end{equation}
A straightforward calculation gives
\begin{align}\label{Gib}
\begin{split}
G_{ij} &= 
g_{ij} - 2u_{\alpha}\langle S_{f_{\alpha}}(\partial_i), 
\partial_j\rangle_g +  
u_{\alpha}u_{\beta}\langle 
S_{f_{\alpha}}(\partial_i), S_{f_{\beta}}(\partial_j)\rangle_g
+  u_{\alpha}u_{\beta}\langle 
\nabla_{\partial i}^{\bot}f_{\alpha}, 
\nabla_{\partial j}^{\bot}f_{\beta}\rangle_g,
\\
G_{i\beta} &= 
u_{\alpha}\langle \nabla_{\partial i}^{\bot}f_{\alpha}, 
f_{\beta}\rangle_g,
\\
G_{\alpha\beta} &= \delta_{\alpha\beta},
\end{split}
\end{align}
where $\partial_i := \frac{\partial}{\partial x_i}$; $\langle\,\,,\,\,\rangle_g$ is the inner product induced by the Riemannian metric;
$g_{ij} dx_idx_j$ is the Riemannian metric on
$\Sigma$; and $S$ is the second fundamental form of~$\Sigma$. 

An \emph{$a$-tube} $V=V(a)$ of $\Sigma$ is the set of all points 
in $\mathbb R^{\dmn+\cdmn}$ whose distance to $\Sigma$ is less than $a$. 
If $\Sigma$ has bounded second fundamental form, 
then there exists an $a>0$ such that $V\subset \mathbb R^{\dmn+\cdmn}$ 
is an immersion, which follows from~\eqref{Gib}. We have

\begin{Lemma}
Assume that $\Sigma$ is an embedded asymptotically flat submanifold of $\mathbb R^{\dmn+\cdmn}$.
Then for any $R>0$, 
there exists $y\in\Sigma$ such that on $\pi^{-1}(B_y(R))$, 
$p$  is an embedding for $a$ small.  
\end{Lemma}

\begin{proof}  Let $\overline{x_1x_2}$ be the Euclidean distance and $d(x_1,x_2)$ be the geodesic distance on $\Sigma$ for any two points $x_1,x_2\in B_y(R)$. Since the functions $\sigma_\alpha$ in Lemma~\ref{lem1} are of small $\mathcal C^2$ norm, $\Sigma$ is very close to its tangent space on $B_{y}(R)$ for some $y$ far away from a fixed point $o$.  Moreover, we have
\[
1/2 \, \overline{x_1x_2}\leq d(x_1,x_2)\leq 2 \, \overline{x_1x_2}
\]
on $B_y(R)$.
By~\eqref{Gib}, $p$ is an immersion, and there exists a $\delta$ such that if $d(x_1,x_2)<\delta$, then $p(x_1)\neq p(x_2)$. On the other hand, if we choose $a<1/3 C^{-1} \delta$, then
for any $z_1, z_2$ such that $\overline{z_i x_i}<a$ for $i=1,2$ and the line segments of $z_ix_i$ $(i=1,2)$ are orthogonal to $\Sigma$ and $d(x_1,x_2)>\delta$, we have 
\[
\overline{p(z_1)p(z_2)} \geq \overline{p(x_1)p(x_2)}-2a>\frac 13 C^{-1}\delta>0
\]
and hence $p(z_1)\neq p(z_2)$.

\end{proof}

\begin{definition} \label{def3}
We say that~$\Omega$ is an \emph{immersed layer}  
if the restriction  $p: T^\perp\Sigma\to\mathbb R^{\dmn+\cdmn}$ 
is an immersion,  $\Omega\subset V(a)$, and $G$ is the restriction of the Euclidean metric of $\mathbb R^{\dmn+\cdmn}$ to $\Omega$.
\end{definition}
\begin{definition}\label{def4}
An immersed   layer is called \emph{asymptotically flat} if
\begin{enumerate}
\item the second fundamental form of $\Sigma$ goes to zero at infinity;
\item the Fermi coordinates provides the diffeomorphism
\[
\pi^{-1}(B_{y_i}(R_i))=B_{y_i}(R_i)\times\omega,
\]
where $(y_i, R_i)$ satisfy the conclusions of Lemma~\ref{lem1};
\item Let $f_1,\cdots, f_\cdmn$ be the frames of 
$T^\perp\Sigma$ on $B_{y_i}(R_i)$ which define the Fermi coordinates. 
Then the connection $\nabla^\perp$ with respect to  the frames go to zero
on $B_{y_i}(R_i)$;
\item $G$ is the restriction of the Euclidean metric of $\mathbb R^{\dmn+\cdmn}$ to $\Omega$.
\end{enumerate}
\end{definition}
%

%---------------%
\section{Proofs}\label{Sec.proofs}
%---------------%
%

Let us recall the definition of~$-\Delta_G$ 
as the Friedrichs extension on $\sii(\Omega,G)$
of the Laplace-Beltrami operator of $(\Omega,G)$
initially defined on $C_0^\infty(\Omega)$.
The operators~$-\Delta_{G_0}$ on $\sii(\Sigma\times\omega,G_0)$
and~$-\Delta_g$ on $\sii(\Sigma,g)$ are introduced analogously. 
By definition, $-\Delta_g$, $-\Delta_G$ and $-\Delta_{G_0}$
are self-adjoint operators on the respective Hilbert spaces.
The corresponding quadratic forms are denoted by
$h_g$, $h_G$ and $h_{G_0}$, respectively;
by definition, the domains of the forms coincide with
the Sobolev spaces, $H_0^1(\Sigma,g)$, $H_0^1(\Omega,G)$
and $H_0^1(\Sigma\times\omega,G_0)$, respectively.
The norm and inner product in $\sii(\Sigma,g)$
are denoted by $\|\cdot\|_g$ and $(\cdot,\cdot)_g$, respectively;
the same subscript convention is used for the other Hilbert spaces as well.
Finally, we use the notations
$$
  |\nabla_{\!g} \psi|_g 
  := \sqrt{\overline{\partial_i\psi} \, g^{ij} \,\partial_j\psi}  
  \qquad \mbox{and} \qquad
  \|\nabla_{\!g} \psi\|_g 
  := \| |\nabla_{\!g} \psi|_g \|_g
  \,, 
$$
and similarly for $(\Omega,G)$ and $(\Sigma\times\omega,G_0)$.

\begin{proof}[Proof of~\eqref{EssSpec0}.]
Let us begin with a more detailed proof of~\eqref{EssSpec0}. 
The block form~\eqref{block} leads to the decoupling
\begin{equation*}
  -\Delta_{G_0} = -\Delta_g \otimes 1 + 1 \otimes (-\Delta_D^\omega)
  \qquad \mbox{on} \qquad
  \sii(\Sigma,g) \otimes \sii(\omega) 
  \,.
\end{equation*}
Employing \cite{RS1}*{Thm.~VIII.33}
and the discreteness of the spectrum of~$-\Delta_D^\omega$,
we arrive at the spectrum decomposition (\cf~\eqref{exp1})
\begin{equation*}
  \sigma(-\Delta_{G_0}) = 
  \overline{
  \sigma(-\Delta_g) + \sigma(-\Delta_D^\omega)
  }
  = \overline{
  \bigcup_{k=1}^\infty 
  \big(
  E_k+\sigma(-\Delta_g)
  \big)
  }
  \,.
\end{equation*}
Finally, Lu and Zhou proved in~\cite{Lu-Zhou_2011}
that if the Ricci curvature of~$\Sigma$
vanishes at infinity, then~\eqref{exp2} holds. 
Note that~$-\Delta_g$ is a non-negative operator,
so that~\eqref{exp2} coincides with the total spectrum of~$-\Delta_g$.  
Consequently, if~$\Sigma$ is asymptotically flat, 
\begin{equation*}
  \sigma(-\Delta_{G_0}) = [E_1,\infty)
  \,.
\end{equation*}
This establishes~\eqref{EssSpec0} because intervals 
have no isolated points.
\hfill

\end{proof}

\begin{proof}[Proof of Theorem~\ref{thm3}.]
Let us now turn to the proof of Theorem~\ref{thm3}, which 
is similar to that of~\cite{LL2}*{Thm.~1}.

By Lemma~\ref{lem1}, let $(x_1,\cdots,x_{\dmn}, u_1,\cdots, u_{\dmn+\cdmn})$ be the 
Fermi coordinates of $\pi^{-1}(B_y(R))$ where $y\in\Sigma$ and $R$ is sufficiently large.
Let
\[
\begin{aligned}
\tilde G_{ij}
&:=g_{ij} - 2u_{\alpha}\langle S_{f_{\alpha}}(\partial_i), 
\partial_j\rangle_g+u_{\alpha}u_{\beta}\langle 
S_{f_{\alpha}}(\partial_i), S_{f_{\beta}}(\partial_j)\rangle_g, 
\\
\tilde{G}_{i\beta} &:= 0
\\
\tilde{G}_{\alpha\beta} &:= \delta_{\alpha\beta}.
\end{aligned}
\]
Note that $\tilde G$ is a $T^*(\Sigma)^{\otimes 2}$ symmetric tensor on $\Omega$.
Then by~\cite{LL2}*{Lemmata 1,2}, we have
\[
|\nabla_{\!G} \psi|_G^2\geq\sum_\alpha \left|\frac{\partial\psi}{\partial u_\alpha}\right|^2,
\]
for every smooth function~$\psi$ on~$\Omega$
and $\det G=\det\tilde G$.

Since $\Sigma$ is asymptotically flat, for any $\eps>0$, 
there exists a compact set $K$ outside which 
\[
(1+\eps)\det g\geq \det\tilde G\geq (1-\eps)\det g.
\]
Using the above inequality and by the Poincar\'e inequality on $\omega$, we  have
\[
\int_\Omega|\nabla_{\!G} \psi|_G^2 \, \det \tilde G
\geq \frac{1-\eps}{1+\eps}\,E_1\int_\Omega |\psi|^2 \det \tilde G.
\]
Thus if the support of $\psi$ is outside $K$, we have
\[
h_G[\psi] \equiv
\int_\Omega|\nabla_{\!G} \psi|_G^2 \, \det G
\geq \frac{1-\eps}{1+\eps} \,E_1\int_\Omega |\psi|^2\det G
\equiv \frac{1-\eps}{1+\eps} \,E_1 \, \|\psi\|_G^2
\,. 
\]
By a Persson-type result \cite{CFKS-2nd}*{Thm.~3.12}
(alternatively one can use a Neumann bracketing argument
in the spirit of~\cite{DEK2}*{Thm.~4.1}), it follows that 
$$
  \inf\sigma_\mathrm{ess}(-\Delta_G) 
  \geq \frac{1-\eps}{1+\eps} \, E_1
  \,.
$$
Theorem~\ref{thm3} follows by sending~$\eps$ to zero.\end{proof}

Before establishing Theorem~\ref{thm1},
we state the following well-known result without proof.
\begin{prop}\label{Prop.Euclid} 
For any $\lambda\geq 0$, 
there exists a sequence of nonzero  smooth functions $\xi_j$ in $\mathbb R^\dmn$ 
with disjoint compact support such that
\[
\|\xi_j\|_{L^2(\Real^\dmn)}=1,\quad 
\|\Delta \xi_j+\lambda \xi_j\|_{L^2(\Real^\dmn)}\to 0
\]
as $j\to\infty$, 
where $\Delta$ is the Laplacian of  the Euclidean space $\mathbb R^\dmn$.
\end{prop}
\noindent
This proposition basically tells us that the singular sequence 
of Theorem~\ref{Thm.Weyl} can be chosen in a particular way
for the Euclidean Laplacian.

\begin{proof}[Proof of Theorem~\ref{thm1}.]
Let $\lambda\geq E_1$. Let $(y_j, R_j)$ be the sequences defined 
in Definition~\ref{def1}. Let $\xi_j$ be the sequence of smooth functions 
defined in the above proposition with $\lambda$ replaced by $\lambda-E_1$. 
Without loss of generality, we may assume that the support of $\xi_j$ 
is contained in the ball of radius $R_j$ in the Euclidean space.  
Since the exponential map at $y_i$ 
in Lemma~\ref{lem1} provides a diffeomorphism 
up to the ball of radius $R_j$, $\xi_j$ 
can also be viewed as smooth  functions on $B_{y_j}(R_j)$ and hence on $\Sigma$.

Let $\sigma_1$ be the first eigenfunction of $\omega$. 
Then the sequence $\psi_j:=\xi_j\otimes\sigma_1$ 
is a sequence of functions 
on $\Sigma\times\omega$. By Definition~\ref{def1}, 
these functions can be viewed as smooth functions on $\Omega$ as well. 
We choose the sequence $\phi_j:=\psi_j/\|\psi_j\|_G$ as the singular sequence. 
It is apparent that the sequence satisfies 
properties~1 and~3 of Theorem~\ref{Thm.Weyl.bis}, 
where the operator~$H$ in that theorem is taken for $-\Delta_G$.

By Definition~\ref{def1},
there exists a positive sequence $\eps_j\to 0$ 
such that on each $\pi^{-1}(B_{y_j}(R_j))$
\begin{equation}\label{impo-3}
  (1-\eps_j) G_0 \leq G \leq (1+\eps_j) G_0.
\end{equation}
In order to verify~(2) of Theorem~\ref{Thm.Weyl.bis}, we establish the identity
\begin{equation*}
\begin{aligned}
  \big( \varphi,(-\Delta_G-\lambda)\phi_j \big)_G
  =\ & \big( \varphi,(-\Delta_{G_0}-\lambda)\phi_j \big)_{G_0}
  + \big(\partial_M\varphi, G^{MK} A_{K}^{\ N} \, \partial_N\phi_{j}\big)_G
  - \lambda \, \big(\varphi, B \phi_{j}\big)_G
\end{aligned}
\end{equation*}
for any $\phi_j$, where $\varphi$ is a 
smooth function on $\Omega$ with compact support and
$$
  A_{K}^{\ N} := \delta_K^{N} - (\det G)^{-1/2} \, (\det G_0)^{1/2} \, G_{KL} \, G_0^{LN}
  \qquad \mbox{and} \qquad
  B := 1 - (\det G)^{-1/2} \, (\det G_0)^{1/2}
  \,.
$$
Consequently,

\begin{equation*}
\begin{aligned}
  \big|\big( \varphi,(-\Delta_G-\lambda)\phi_j \big)_G\big|
  \leq \ &\big|\big( \varphi,(-\Delta_{G_0}-\lambda)\phi_j \big)_{G_0}\big|
  \\
  &+ \|\nabla_{\!G}\varphi\|_G \, 
  \Big(\sup_{\Sigma \setminus B_{y_j}(R_j)} |A| \Big) 
  \, \|\nabla_{\!G}\phi_{j}\|_G
  \\
  &+ \lambda \, \big\|\varphi\|_G  \, 
  \Big(\sup_{\Sigma \setminus B_{y_j}(R_j)} |B|\Big) 
  \, \|\phi_{j}\big\|_G
  \,,
\end{aligned}
\end{equation*}
where~$|A|$ denotes the matrix (operator) norm of 
$A : \Real^{\dmn+\cdmn} \to \Real^{\dmn+\cdmn}$ (defined by $(A_K^N)$) 
and~$|B|$ is just the absolute value of the function~$B$. 
By~\eqref{impo-3}, both $A^N_K$ and $B$ tend to zero. 
Moreover,   
\[
\|\nabla_{\!G}\phi_j\|_G
\leq C \|\nabla_{\!G_0}\phi_j\|_{G_0}
\leq C\big(\|\xi_j\|_g+\|\nabla_g \xi_j\|_g\big)
\,,
\qquad \|\phi_j\|_G\leq C\|\phi_j\|_{G_0}
\,.
\]
Here and in the sequel, $C$~denotes a $j$-independent constant 
that may vary from line to line.

By Lemma~\ref{lem1},
$(1-\eps_j) g_{\mathbb R^\dmn}\leq g\leq (1+\eps_j) g_{\mathbb R^\dmn}$  
on $B_{y_j}(R_j)$ for the sequence $\eps_j\to 0$. Therefore
\[
\|\nabla_{\!g} \xi_j\|_g+\|\xi_j\|_g\leq C\|\xi_j\|_{H^1(\mathbb R^\dmn)}\leq C.
\]
Thus we have
\[
 \big|\big( \varphi,(-\Delta_G-\lambda)\phi_j \big)_G\big|
  \leq \big|\big( \varphi,(-\Delta_{G_0}-\lambda)\phi_j \big)_{G_0}\big|+o(1)
\]
where $o(1)$ tends to zero as $j\to\infty$.

To complete the proof, we first observe that 
\[
\big( \varphi,(-\Delta_{G_0}-\lambda)\phi_j \big)_{G_0}
= \big(\tilde \varphi, \Delta_g \xi_j +(\lambda-E_1)\xi_j\big)_g,
\]
where $\tilde\varphi$ is obtained from $\varphi$ by integrating over  the fiber.
 Similar to the above, we let
\begin{equation*}
\tilde A_k^l:=\delta_k^l-(\det g)^{-1/2} g_{kl}
\qquad\mbox{and}\qquad
\tilde B:=1-(\det g)^{-1/2} \,.
\end{equation*}
%
%\begin{align*}
%& \tilde A_k^l=\delta_k^l-|g|^{-1/2} g_{kl},\\
%&\tilde B=1-|g|^{-1/2}.
%\end{align*}
Then we have
\[
\big(\tilde\varphi, \Delta_g f_j +(\lambda-E_1)\xi_j\big)_g
=\big(\tilde\varphi, \Delta_g f_j +(\lambda-E_1)f_j\big)_{L^2({\mathbb R^\dmn})}
+ \big(\partial_i\tilde\varphi, G^{ik} \tilde A_{k}^{l} \, \partial_l f_{j}\big)_g
  - (\lambda-E_1) \, \big(\tilde \varphi, \tilde B \phi_{j}\big)_g.
  \]
Since the second fundamental form of $\Sigma$ tends to zero, 
so does $\tilde A_k^l,\tilde B$. By an argument similar to the above, 
we conclude that the left hand side of the above inequality tends to zero 
which  completes the proof.
\end{proof}

\begin{proof}[Proof of Theorem~\ref{thm2}.]
It follows from Definition~\ref{def4}, 
Equation~\eqref{Gib} and Lemma~\ref{lem1}
that the metric~$G$ is a local deformation of~$G_0$. 
Consequently, Theorems~\ref{thm3} and~\ref{thm1} 
yield Theorem~\ref{thm2} as a direct corollary. 
\end{proof}

\appendix
%------------------------------------------------%
\section{The Weyl criterion for quadratic forms}\label{App}
%------------------------------------------------%
%
Let~$H$ be a self-adjoint operator on a Hilbert space~$\Hilbert$.
The norm and inner product in~$\Hilbert$ are respectively
denoted by~$\|\cdot\|$ and $(\cdot,\cdot)$. 
The classical Weyl criterion can be stated as follows.
\begin{Theorem}[Classical Weyl's criterion]\label{Thm.Weyl}
A point~$\lambda$ belongs to~$\sigma(H)$ if, and only if,
there exists a sequence $\{\psi_n\}_{n \in \Nat} \subset \Dom(H)$
such that
\begin{enumerate}
\item
$
  \forall n\in\Nat, \quad
  \|\psi_n\|=1
$\,,
\item
$
  (H-\lambda)\psi_n \xrightarrow[n\to\infty]{} 0
$
\quad in \quad $\Hilbert$.
\end{enumerate}
Moreover, $\lambda$ belongs to~$\sigma_\mathrm{ess}(H)$ if, and only if,
in addition to the above properties
\begin{enumerate}
\setcounter{enumi}{2}
\item
$
  \psi_n \xrightarrow[n\to\infty]{w} 0
$
\quad in \quad $\Hilbert$.
\end{enumerate}
\end{Theorem}

The fact that the sequence satisfying the properties~1 and~2 
ensures that~$\lambda$ belongs to the spectrum of~$H$ 
is true for a general (not necessarily self-adjoint) closed operator.
It is a consequence of the spectral theorem that
the properties~1 and~2 represent a necessary condition too.
The sequence satisfying the items~1--3
is called a \emph{singular sequence}.
We refer to~\cite{Weidmann}*{Sec.~7.4}  for a proof of Theorem~\ref{Thm.Weyl}
and other equivalent statements. 

If one wants to use Theorem~\ref{Thm.Weyl} in the ``$\Leftarrow$'' sense,
\ie~to construct the required sequence in order to prove that
a point belongs to the (essential) spectrum, 
a defect of Theorem~\ref{Thm.Weyl} consists in that 
it requires to take the sequence from the operator domain
and to establish the strong limit of item~2.
In applications, however, it is sometimes useful to work with
associated quadratic forms instead of operators.

There is a one-to-one correspondence between closed symmetric forms
and self-adjoint operators which are bounded from below. 
Under the latter restriction, we may without loss of generality
assume that the operator~$H$ is non-negative
(since adding a sufficiently large constant will
make the operator positive).
Let $\Hilbert_{+1}$ denote the form domain $\Dom(H^{1/2})$  of~$H$, 
equipped with its graph norm 
$$
  \|\cdot\|_{+1} 
  := \big\|(H+1)^{1/2} \cdot\big\|
  = \sqrt{h[\cdot]+\|\cdot\|^2}
  \,,
$$
where~$h$ is the quadratic form associated with~$H$.  
Define~$\Hilbert_{-1}$ to be the dual space of~$\Hilbert_{+1}$
and let us denote the pairing
between~$\Hilbert_{+1}$ and~$\Hilbert_{-1}$ also by $(\cdot,\cdot)$. 
Then we have the embeddings
\begin{equation}\label{hooks}
  \Hilbert_{+1} \hookrightarrow \Hilbert \hookrightarrow \Hilbert_{-1}
\end{equation}
and~$\Hilbert_{-1}$ can be thought as the closure of~$\Hilbert$ 
in the norm
\begin{equation}\label{dual.norm}
  \|\cdot\|_{-1} 
  := \big\|(H+1)^{-1/2} \cdot\big\|
  = \sup_{\phi\in\Hilbert_{+1}\setminus\{0\}}
  \frac{|(\phi,\cdot)|}{\,\|\phi\|_{+1}} 
  \,.
\end{equation}

Now we are in a position to state an improved version
of the Weyl criterion in the case of operators bounded from below. 
\begin{Theorem}[Weyl's criterion for quadratic forms]\label{Thm.Weyl.bis}
Let~$H$ be non-negative.
A point~$\lambda$ belongs to~$\sigma(H)$ if, and only if,
there exists a sequence $\{\psi_n\}_{n \in \Nat} \subset \Dom(H^{1/2})$
such that
\begin{enumerate}
\item
$
  \forall n\in\Nat, \quad
  \|\psi_n\|=1
$\,,
\item
$
  (H-\lambda)\psi_n \xrightarrow[n\to\infty]{} 0
$
\quad in \quad $\Hilbert_{-1}$.
\end{enumerate}
Moreover, $\lambda$ belongs to~$\sigma_\mathrm{ess}(H)$ if, and only if,
in addition to the above properties
\begin{enumerate}
\setcounter{enumi}{2}
\item
$
  \psi_n \xrightarrow[n\to\infty]{w} 0
$
\quad in \quad $\Hilbert$.
\end{enumerate}%
\end{Theorem}
\begin{proof}
By Theorem~\ref{Thm.Weyl} and~\eqref{hooks}, 
the direction ``$\Rightarrow$'' is trivial.
%We prove the opposite direction ``$\Leftarrow$'' by contradiction.
Let us assume the existence of a sequence $\{\psi_n\}_{n \in \Nat}$
satisfying~1 and~2, but $\lambda \not\in \sigma(H)$. 
Then  
$
  E(\lambda+\epsilon) - E(\lambda-\epsilon) = 0 
$ 
for some $\epsilon\in(0,1)$, 
where~$E$ denote the spectral family of~$H$.  
We have
\begin{equation*}
  \big\|(H-\lambda)\psi_n\big\|_{-1}^2 
  = \big\|(H-\lambda)(H+1)^{-1/2}\psi_n\big\|^2 
  = \int_{\Real} \frac{(t-\lambda)^2}{t+1} \, d\|E(t)\psi_n\|^2
  \geq \frac{\epsilon^2}{\lambda+\epsilon+1} \, \|\psi_n\|^2
  \,,
\end{equation*}
where the inequality follows from the fact that
$$
  \frac{(t-\lambda)^2}{t+1}
  \geq \min\left\{
  \frac{\epsilon^2}{\lambda-\epsilon+1},\frac{\epsilon^2}{\lambda+\epsilon+1}
  \right\}
  = \frac{\epsilon^2}{\lambda+\epsilon+1}
$$
almost everywhere relative to the measure induced by $\|E(t)\psi_n\|^2$.
This is a contradiction with~1 and~2;
hence, 1~and~2 implies $\lambda\in\sigma(H)$.   
It remains to show that $\lambda\in\sigma_\mathrm{ess}(H)$
if in addition~3 holds. 
By contradiction, let as assume that 
$
  \dim \Ran\big(E(\lambda+\epsilon) - E(\lambda-\epsilon)\big) < \infty
$
for some $\epsilon\in(0,1)$, \ie,
the projection
$
  E(\lambda+\epsilon) - E(\lambda-\epsilon)
$ 
is compact.
Then
$
  \big(E(\lambda+\epsilon) - E(\lambda-\epsilon)\big)\psi_n 
  \to 0
$ 
strongly as $n\to\infty$. 
Consequently,
\begin{align*}
  \big\|(H-\lambda)\psi_n\big\|_{-1}^2 
  = \int_{\Real} \frac{(t-\lambda)^2}{t+1} \, d\|E(t)\psi_n\|^2
  &\geq \frac{\epsilon^2}{\lambda+\epsilon+1}
  \left[
  \int_{\Real} d\|E(t)\psi_n\|^2
  - \int_{\lambda-\epsilon}^{\lambda+\epsilon} d\|E(t)\psi_n\|^2
  \right]
  \\
  &= \frac{\epsilon^2}{\lambda+\epsilon+1}
  \left[
  \|\psi_n\|^2
  - \big\|\big(E(\lambda+\epsilon) - E(\lambda-\epsilon)\big)\psi_n \big\|^2
  \right]
  \,,
\end{align*}
and thus
$$
  \liminf_{n\to\infty} \big\|(H-\lambda)\psi_n\big\|_{-1}^2 
  \geq \frac{\epsilon^2}{\lambda+\epsilon+1}
  > 0 
  \,.
$$
This is a contradiction with~2.
\end{proof}
\begin{Remark}
Theorem~\ref{Thm.Weyl.bis} can be found (without proof) in \cite{DDi}. 
D.K. is grateful to V.~Iftimie \cite{Iftimie-private} 
for letting him known about this version
of Weyl's criterion and the proof, 
which is in fact a straightforward modification
of the classical proof of Theorem~\ref{Thm.Weyl} 
(\cf~\cite{Weidmann}*{proofs of Thms.~7.22 and 7.24}).
\end{Remark}
%

%----------------------------%
\subsection*{Acknowledgement}
%----------------------------%
The work has been partially supported 
by RVO61389005, the GACR grant No.\ P203/11/0701,
and the NSF award DMS-1206748.

%\newpage
%--------------%
% BIBLIOGRAPHY %
%--------------%
%
%\addcontentsline{toc}{section}{References}
%

%\listofchanges

\end{document}